\documentclass[a4paper,12pt]{article}

\usepackage[utf8]{inputenc}
\usepackage[english]{babel}
\usepackage{hyperref}
\usepackage{amsmath,amsthm,enumerate}
\usepackage{enumitem} 
\usepackage{amssymb}
\usepackage{graphicx, color}
\usepackage{epsfig}
\usepackage{tikz}
\usetikzlibrary{calc,decorations.pathmorphing,decorations.text}
\usepackage{subcaption}
\usepackage{refcount}
\usepackage[hmargin=1cm,vmargin=2cm]{geometry} 
\usepackage{mathtools, braket}
\usepackage{authblk}
\usepackage{centernot}

\theoremstyle{plain}
\newtheorem{thm}{Thm}[section]

\newtheorem{theorem}[thm]{Theorem}
\newtheorem{lemma}[thm]{Lemma}
\newtheorem{corollary}[thm]{Corollary}

\newtheorem{conjecture}[thm]{Conjecture}

\newtheorem{definition}[thm]{Definition}

\newenvironment{proof*}{\noindent\emph{Proof of Theorem~}}{\hfill$\Box$}

\newcommand{\spto}{\ensuremath{\stackrel{s.p.}{\longrightarrow}}}
\newcommand{\overbar}[1]{\mkern 1.5mu\overline{\mkern-1.5mu#1\mkern-1.5mu}\mkern 1.5mu}
\renewcommand{\pod}[1]{\allowbreak\mathchoice
	{\if@display \mkern 0mu\else \mkern 0mu\fi (#1)}
	{\if@display \mkern 0mu\else \mkern 0mu\fi (#1)}
	{\mkern 1mu(\mathrm{mod}\mkern 4mu #1)}
	{\mkern 0mu(#1)}
}

\usepackage[misc,geometry]{ifsym}

\usepackage[color=green!30]{todonotes}
\usepackage{caption}
\usetikzlibrary{matrix}
\usetikzlibrary{decorations.markings}
\tikzstyle{vertex}=[circle, draw, fill=black!50,
inner sep=0pt, minimum width=4pt]
\tikzset{->-/.style={decoration={
			markings,
			mark=at position .5 with {\arrow{>}}},postaction={decorate}}}

\tikzstyle{bigblue}=[color=blue, very thick, >=stealth]
\tikzstyle{lightblue}=[color=blue, thin, >=stealth]

\tikzstyle{bigred}=[color=red, very thick, >=stealth]
\tikzstyle{lightred}=[color=red, thin, >=stealth]

\tikzstyle{biggreen}=[color=black!30!green, very thick, >=stealth]
\tikzstyle{lightgreen}=[color=black!30!green,  thin, >=stealth]

\title{Signed bipartite circular cliques and\\ a bipartite analogue of Gr\"{o}tzsch's theorem}

\date{\today} 

\author{Reza Naserasr} 
\author{Zhouningxin Wang (\Letter)}

\affil{Université de Paris, IRIF, CNRS, F-75013 Paris, France. 
	
	Email addresses: \{reza, wangzhou4\}@irif.fr.}

\begin{document}

\maketitle

\abstract{\noindent The notion of the circular coloring of signed graphs is a recent one that simultaneously extends both notions of the circular coloring of graphs and $0$-free coloring of signed graphs. A circular $r$-coloring of a signed graph $(G, \sigma)$ is to assign points of a circle of circumference $r$, $r\geq 2$, to the vertices of $G$ such that vertices connected by a positive edge are at circular distance at least $1$ and vertices connected by a negative edge are at circular distance at most $\frac{r}{2}-1$. The infimum of all $r$ for which $(G, \sigma)$ admits a circular $r$-coloring is said to be the circular chromatic number of $(G, \sigma)$ and is denoted by $\chi_c(G, \sigma)$. For any rational number $r=\frac{p}{q}$, two notions of circular cliques are presented corresponding to the edge-sign preserving homomorphism and the switching homomorphism. 

\noindent It is also shown that the restriction of the study of circular chromatic numbers to the class of signed bipartite simple graphs already captures the study of circular chromatic numbers of graphs via basic graph operations, even though the circular chromatic number of every signed bipartite graph is bounded above by $4$. 

\noindent In this work, we consider the restriction of the circular chromatic number to this class of signed graphs and construct signed bipartite circular cliques with respect to both notions of homomorphisms. We then present reformulations of the $4$-Color Theorem and the Gr\"otzsch theorem. As a bipartite analogue of Gr\"otzsch's theorem, we prove that every signed bipartite planar graph of negative girth at least $6$ has circular chromatic number at most $3$.}

\section{Introduction}
\setlength{\parindent}{0pt}
A \emph{signed graph} is a graph $G$ (allowing loops and multi-edges) together with an assignment $\sigma: E(G) \rightarrow \{+, -\}$, denoted by $(G, \sigma)$. We note that signed graphs are allowed to have multi-edges, but we consider multi-edges only if they are of different signs. The signed graph on two vertices connected by two parallel edges of different signs is called \emph{digon}. Furthermore, unless specified, graphs are considered to have no loop.
The \emph{sign of a closed walk} of $(G, \sigma)$ is the product of signs of all its edges (allowing repetition). Given a signed graph $(G, \sigma)$ and a vertex $v$ of $(G, \sigma)$, a \emph{switching at $v$} is to switch the signs of all the edges incident to $v$. We say a signed graph $(G, \sigma')$ is \emph{switching equivalent} to $(G, \sigma)$ if it is obtained from $(G, \sigma)$ by a series of switchings at vertices. In this case, we say the signature $\sigma'$ is \emph{equivalent} to $\sigma$. It has been proved in \cite{Za82} that two signed graphs $(G, \sigma_1)$ and $(G, \sigma_2)$ are switching equivalent if and only if they have the same set of negative cycles.

A \emph{(switching) homomorphism} of a signed graph $(G, \sigma)$ to $(H, \pi)$ is a mapping of $V(G)$ and $E(G)$ to $V(H)$ and $E(H)$ (respectively) such that the adjacencies, the incidences and the signs of the closed walks are preserved. When there exists such a homomorphism, we write $(G, \sigma)\to (H, \pi)$. A homomorphism of $(G, \sigma)$ to $(H, \pi)$ is said to be \emph{edge-sign preserving} if it, furthermore, preserves the signs of the edges. When there exists an edge-sign preserving homomorphism of $(G, \sigma)$ to $(H, \pi)$, we write $(G, \sigma)\spto (H, \pi)$. The connection between these two kinds of homomorphisms is established as follows: Given two signed graphs $(G, \sigma)$ and $(H, \pi)$, $(G, \sigma)\to (H, \pi)$ if and only if there exists a signature $\sigma'$ which is equivalent to $\sigma$, such that $(G, \sigma')\spto (H, \pi)$. An equivalent reformulation of this is through the following definition.

\begin{definition}
Given a signed graph $(G, \sigma)$, the Double Switch Graph of it, denoted DSG$(G, \sigma)$, is the signed graph built from two disjoint copies $(G_1,\sigma_1)$ and $(G_2,\sigma_2)$ of $(G, \sigma)$ by adding the following set of edges in between. If $xy$ is a positive (resp. negative) edge of $(G, \sigma)$, then $x_1y_2$ and $x_2y_1$ are negative (resp. positive) edges of DSG$(G, \sigma)$. Here $x_1,x_2,y_1,y_2$ are representing copies of $x$ and $y$ in $G_1$ and $G_2$ in the most natural way.
\end{definition}

The connection between the two notions of homomorphisms is as follows.

\begin{theorem}{\rm \cite{BG09}}
A signed graph $(G,\sigma)$ admits a switching homomorphism to $(H, \pi)$ if and only if it admits an edge-sign preserving homomorphism to DSG$(H,\pi)$. 
\end{theorem}

We note that in building DSG$(H,\pi)$ for the purpose of this theorem, if for a vertex $x$ there already exists a vertex $x'$ which is like obtained from $x$ by a switching, then we do not need to add a copy of $x$. 
When $(H, \pi)$ has the property that for each vertex $x$ of it, there is a such a vertex $x'$ (which is like a switched copy of $x$), then a signed graph $(G, \sigma)$ admits a homomorphism to $(H, \pi)$ if and only if it also admits an edge-sign preserving homomorphism to $(H,\pi)$. That is because if a vertex $u$ of $(G,\sigma)$ is mapped to the vertex $x$ of $(H, \pi)$ after a switching, then one may instead map $u$ to $x'$ without a switching.
We note that some of the circular cliques that we will discuss here are of this form.

Observe that the parity of the lengths and the signs of closed walks are preserved by a homomorphism. Given a signed graph $(G, \sigma)$ and an element $ij\in \mathbb{Z}^2_2$, we define $g_{ij}(G,\sigma)$ to be the length of a shortest closed walk whose number of negative edges modulo $2$ is $i$ and whose length modulo $2$ is $j$. When there exists no such closed walk, we say $g_{ij}(G, \sigma)=\infty$. In general, the length of the shortest negative closed walk of a signed graph is said to be its \emph{negative girth}. By the definition of the homomorphism of signed graphs, we have the following no-homomorphism lemma.
 
\begin{lemma}\label{lem:no-hom}{\em \cite{NSZ21}}{\rm [No-homomorphism lemma]}
If $(G, \sigma)\to (H, \pi)$, then $g_{ij}(G, \sigma)\geq g_{ij}(H, \pi)$ for each $ij\in \mathbb{Z}^2_2$. 
\end{lemma}

\subsection{Homomorphisms of signed bipartite graphs}

Given a graph $G$, we construct a signed bipartite graph $S(G)$ as follows: The vertex set of $S(G)$ consists of $V(G)$ and $\{x_{uv}, y_{uv}\mid uv \in E(G)\}$; For the edge set, we join each of $x_{uv}$ and $y_{uv}$ to each of $u$ and $v$, and for the signature, we assign signs such that each $4$-cycle $ux_{uv}vy_{uv}$ is negative. Intuitively, in constructing $S(G)$, we replace each edge $uv$ of $G$ with a negative $4$-cycle. We note there is more than one choice of signature here, and moreover, having vertices already labeled, not every two such signatures are switching equivalent. That which of the two sides of a $4$-cycle is chosen to be negative makes a difference here. However, up to a switching isomorphism, any two signature choices of $S(G)$ are the same.

The first easy observation on $S(G)$ is that it is a signed bipartite graph where in one part all vertices are of degree $2$. This construction was introduced in \cite{NRS15} where the next two theorems are proved in order to show the importance of the study of homomorphisms of signed bipartite graphs.

\begin{theorem}{\rm \cite{NRS15}}\label{thm:S(G)->S(H)}
Given graphs $G$ and $H$, $G\to H$ if and only if $S(G)\to S(H)$.
\end{theorem}

\begin{theorem}{\rm \cite{NRS15}}\label{thm:X->S(G)}
Given a graph $G$, we have the followings.
\begin{itemize}
	\item $\chi(G)\leq 2$ if and only if $S(G)\to (K_{2,2}, e)$;
	\item $\chi(G)\leq k$ if and only if $S(G)\to (K_{k,k}, M)$ for $k\geq 3$.
\end{itemize}
\end{theorem}

As the problem of mapping signed graphs to $(K_{k,k},M)$ could capture the problem of the coloring of ordinary graphs, signed graphs $(K_{k,k},M)$ are of special interests in the study of the homomorphism of signed graphs. In particular, when $k=4$, we have a restatement of $4$-Color Theorem as follows.

\begin{theorem}\label{thm:Restated4CT}
	For any planar graph $G$, $S(G)\to (K_{4,4},M)$. 
\end{theorem}

Notice that the planarity is preserved when we construct $S(G)$ from a planar graph $G$. 
Moreover, based on an edge-coloring result of B. Guenin \cite{G03} which in turn is based on the $4$-Color Theorem, the following strengthening of the 4-Color Theorem has been proved in \cite{NRS15}: Every signed planar graph $(G, \sigma)$ satisfying that $g_{ij}(G, \sigma)\geq g_{ij} (K_{4,4}, M)$ for $ij\in \mathbb{Z}^2_2$ maps to $(K_{4,4},M)$. A signed graph $(G, \sigma)$ satisfies the conditions $g_{ij}(G, \sigma)\geq g_{ij} (K_{4,4}, M)$ if and only if its underlying graph is bipartite and it has no digon. The next theorem is a reformulation of this strengthening of the 4-Color Theorem.

\begin{theorem}\label{thm:4CTStrengthened}
Every signed bipartite planar simple graph admits a homomorphism to $(K_{4,4,} M)$.
\end{theorem}

The $4$-Color Theorem states that every planar graph admits a homomorphism to $K_4$, while the Gr\"otzsch theorem states that every planar graph of girth at least $4$ maps to $K_3$, noting that $K_3$ is a subgraph of $K_4$. Motivated by this observation, Theorem~\ref{thm:X->S(G)}, and Theorem~\ref{thm:4CTStrengthened}, it is quite natural to ask for which families of signed bipartite planar graphs admit $(K_{3,3}, M)$ as a homomorphism bound. Observe that here $(K_{3,3},M)$ is a subgraph of $(K_{4,4}, M)$. We further note that, considering Theorem~\ref{thm:X->S(G)}, the problem of mapping signed bipartite planar graphs to $(K_{3,3}, M)$ captures the $3$-coloring problem of (ordinary) planar graphs. In this work, among other results, we prove the following theorem in Section~\ref{sec:CircularChromaticNumber_3} noting that our proof is based on the $4$-Color Theorem.

\begin{theorem}\label{thm:main}
Every signed bipartite planar graph of negative girth at least $6$ admits a homomorphism to $(K_{3,3}, M)$. Moreover, the girth condition is best possible.
\end{theorem}

\section{Circular chromatic number of signed graphs}\label{sec:X_c}

Given a signed graph $(G, \sigma)$ and a positive integer $k$, a \emph{$0$-free $2k$-coloring} of $(G, \sigma)$ (introduced in \cite{Za82}) is a mapping $f: V(G) \to \{\pm 1, \pm 2, \ldots, \pm k\}$ such that for any edge $e=uv$, $f(u) \ne \sigma(e) f(v)$. The notion of the circular coloring of signed graphs defined in \cite{NWZ21} is a common extension of circular coloring of graphs and $0$-free $2k$-coloring of signed graphs.

For a real number $r \ge 1$, let $C^r$ be a circle of circumference $r$. For two points $x, y$ on $C^r$, the {\rm distance} between $x$ and $y$ on $C^r$, denoted $d_{C^r}(x,y)$, is the length of the shorter arc of $C^r$ connecting $x$ and $y$. For each point $x$ on $C^r$, the {\em antipodal} of $x$, denoted by $\bar{x}$, is the unique point at distance $\frac{r}{2}$ from $x$. 

Given a real number $r$, a \emph{circular $r$-coloring} of a signed graph $(G, \sigma)$ is a mapping $\varphi: V(G) \to C^r$ such that 
\begin{itemize}
\item for each positive edge $uv$ of $(G, \sigma)$,  $d_{C^r}(\varphi(u),  \varphi(v)) \geq 1$;
\item for each negative edge $uv$ of $(G, \sigma)$, $d_{C^r}(\varphi(u),  \overbar{\varphi(v)}) \geq 1$.
\end{itemize}
The \emph{circular chromatic number} of a signed graph $(G, \sigma)$ is defined as $$\chi_c(G, \sigma) = \inf \{r \ge 1: (G, \sigma) \text{ admits a circular $r$-coloring}\}.$$

For integers $p \ge 2q > 0$ such that $p$ is even, the \emph{signed circular clique} $K_{p;q}^s$ has the vertex set $[p] = \{0,1,\ldots, p-1\}$, in which $ij$ is a positive edge if and only if  $q \leq |i-j| \leq p-q$ and $ij$ is a negative edge if and only if either $|i-j| \leq \frac{p}{2}-q$ or $|i-j| \geq \frac{p}{2}+q$. Moreover, let $\hat{K}_{p;q}^s$ be the signed subgraph of $K_{p;q}^s$ induced by vertices $\{0,1,\ldots, \frac{p}{2}-1\}$. In this definition loops are allowed, and indeed, by the definition, there will be a negative loop on each vertex but there will be no positive loop. As shown in \cite{NWZ21}, the following statements are equivalent:
\begin{itemize}
\item $(G, \sigma)$ admits a circular $\frac{p}{q}$-coloring;
\item $(G, \sigma)$ admits an edge-sign preserving homomorphism to $K_{p;q}^s$;
\item $(G, \sigma)$ admits a switching homomorphism to $\hat{K}_{p;q}^s$.
\end{itemize}

In other words, in the order induced by edge-sign preserving homomorphism on the class of all signed graphs, the circular chromatic number of a signed graph $(G, \sigma)$ is the smallest value of a rational number $\frac{p}{q}$ such that $(G, \sigma) \spto K_{p;q}^s$. Normally we choose the minimal element (the \emph{core}) of each homomorphically equivalent class to present the class. In such cases then we will choose $K_{p;q}^s$ where $p$ is an even integer and, with respect to this condition, $\frac{p}{q}$ is in its simplest form, e.g. $K^s_{16;5}$ or $K^s_{8;2}$. Observe that in $K_{p;q}^s$ if we apply a switching at a vertex $i$, $i\leq \frac{p}{2}-1$, then we get a copy of the vertex $i+\frac{p}{2}$. Furthermore, $K_{p;q}^s={\rm DSG}(\hat{K}_{p;q}^s)$. Thus $\hat{K}_{p;q}^s$ is a homomorphic image of $K_{p;q}^s$ with respect to the (switching) homomorphism. Furthermore, with the same assumption on $p$ and $q$, $\hat{K}_{p;q}^s$ is a core. 

The next lemma is a straightforward consequence of the transitivity of the homomorphism relation.

\begin{lemma}\label{lem:No-HomByCircularChromatic}
If $(G,\sigma)\to (H, \pi)$, then $\chi_c(G,\sigma) \leq \chi_c(H, \pi)$.
\end{lemma} 

Let $D$ be a digon. It follows immediately that every signed bipartite graph admits an (edge-sign preserving) homomorphism to $D$ and, as $\chi_c(D)=4$, we have an upper bound of $4$ for the circular chromatic number of signed bipartite graphs. However, the restriction of the problem to this subclass of signed graphs is still of high interest as shown by the following result of \cite{NWZ21}.

\begin{theorem}\label{thm:X(G)-->X(S(G))}
Given a graph $G$, we have $$\chi_c(S(G)) = 4-\dfrac{4}{\chi_c(G)+1}.$$
\end{theorem}

That is equivalent to: $\chi_c(G) = \frac{\chi_c(S(G))}{4-\chi_c(S(G))}$. In particular, we have that 
\begin{itemize}
\item $\chi_c(G)\leq 4$ if and only if $\chi_c(S(G))\leq \frac{16}{5}$,
\item $\chi_c(G)\leq 3$ if and only if $\chi_c(S(G))\leq 3$.
\end{itemize}
In the next section, we study the restriction of the circular chromatic number to the class of signed bipartite graphs and especially, we introduce the signed bipartite circular clique.

\subsection{Signed bipartite circular clique $B_{p;q}$}
One may view the class of signed circular cliques $K_{p;q}^s$ or $\hat{K}_{p;q}^s$ as a representation of rational numbers in the homomorphism order of the class of all signed graphs. Then the circular chromatic number of a signed graph $(G, \sigma)$ is determined by the first element of this chain (representing rational numbers) which is larger than $(G,\sigma)$ with respect to the homomorphism order. 

In Theorems~\ref{thm:S(G)->S(H)}, \ref{thm:X->S(G)} and \ref{thm:X(G)-->X(S(G))}, we have seen the importance of the restriction of the study into the subclass $\mathcal{SB}$ of signed bipartite graphs. A natural question to ask is if the homomorphism order restricted to this subclass behaves similarly? More precisely, we would like to know if there is a chain of signed bipartite graphs in the homomorphism order on $\mathcal{SB}$ which plays the role of circular clique? 

We note that no signed circular clique $\hat{K}_{p;q}^s$ or $K_{p;q}^s$ is bipartite. Indeed each vertex in any of these cliques has a negative loop on it. In this section, for $\frac{p}{q}\leq 4$, we introduce a bipartite subgraph of these circular cliques that plays the role of circular clique in the restricted class $\mathcal{SB}$. 
 
\begin{definition}
Given a rational number $\frac{p}{q}$ where $p$ is an even number, $2\leq \frac{p}{q}\leq 4$ and subject to these conditions $\frac{p}{q}$ is in its simplest form, we define the signed graph $B_{p;q}$ to be the following subgraph of $K^s_{p;q}$: The vertex set $[p]=\{0,1, \ldots, p-1\}$ is partitioned to two parts $X$ and $Y$ where $X=\{0,2, \ldots, p-2\}$ and $Y=\{1,3, \ldots, p-1\}$. The edge set is formed by the edges of $K^s_{p;q}$ which have exactly one endpoint in $X$ and another endpoint in $Y$. The signs of edges are also induced by $K^s_{p;q}$. 
\end{definition}

We will show that $B_{p;q}$, which itself is a signed bipartite graph, plays the roll of circular clique in the subclass of signed bipartite graphs. However, this class of signed graphs is partitioned into two subclasses depending on whether $p$ is a multiple of $4$ or it is 2 $\pmod 4$. 

When $p$ is a multiple of $4$, then we will show that $B_{p;q}$ is a circular clique with respect to the edge-sign preserving homomorphism.  It means that any signed bipartite graph of circular chromatic number at most $\frac{p}{q}$ admits an edge-sign preserving homomorphism to $B_{p;q}$. In this case, as we will show, the subgraph induced on the vertices $[\frac{p}{2}]=\{0,1, \ldots, \frac{p}{2}-1\}$ forms the core of $B_{p;q}$ and will play the role of signed bipartite clique with respect to the switching homomorphism. For example, $B_{16;5}$ is depicted in Figure~\ref{fig:B16;5} and its switching core, which is a signed graph on $K_{4,4}$, is depicted in Figure~\ref{fig:hatB16;5}. 

\begin{figure}[htbp]
	\centering
	\begin{minipage}{.45\textwidth}
		\centering
		\begin{tikzpicture}
			[scale=.22]		
			\foreach \i in {1,3,5,7,9,11,13,15}
			{
				\draw[rotate=360-22.5*(\i)] (0, 12) node[circle, fill=lightgray, draw=black!80, inner sep=0mm, minimum size=3.3mm] (x\i){\scriptsize ${\i}$};
			}
			
			\foreach \i in {0, 2,4,6,8,10,12,14}
			{
				\draw[rotate=360-22.5*(\i)] (0, 12) node[circle, draw=black!80, inner sep=0mm, minimum size=3.3mm] (x\i){\scriptsize ${\i}$};
			}

			\foreach \i/\j in {1/6,1/8,1/10,1/12}
			{
				\draw  [bend left=18, line width=0.4mm, blue] (x\i) -- (x\j);
			}
			
			\foreach \i/\j in {1/2,1/4,1/14,1/0}
			{
				\draw  [bend left=18, densely dotted, line width=0.4mm, red] (x\i) -- (x\j);
			}	
			
			\foreach \i/\j in {2/7,2/9,2/11,2/13}
			{
				\draw  [bend left=18, line width=0.4mm, blue] (x\i) -- (x\j);
			}
			
			\foreach \i/\j in {2/3,2/5,2/15}
			{
				\draw  [bend left=18, densely dotted, line width=0.4mm, red] (x\i) -- (x\j);
			}	
			
			\foreach \i/\j in {3/8,3/10,3/12,3/14}
			{
				\draw  [bend left=18, line width=0.4mm, blue] (x\i) -- (x\j);
			}
			
			\foreach \i/\j in {3/4,3/6,3/0}
			{
				\draw  [bend left=18, densely dotted, line width=0.4mm, red] (x\i) -- (x\j);
			}	
			
			\foreach \i/\j in {4/9,4/11,4/13,4/15}
			{
				\draw  [bend left=18, line width=0.4mm, blue] (x\i) -- (x\j);
			}
			
			\foreach \i/\j in {4/5,4/7}
			{
				\draw  [bend left=18, densely dotted, line width=0.4mm, red] (x\i) -- (x\j);
			}	
			
			\foreach \i/\j in {5/10,5/12,5/14,5/0}
			{
				\draw  [bend left=18, line width=0.4mm, blue] (x\i) -- (x\j);
			}
			
			\foreach \i/\j in {5/6,5/8}
			{
				\draw  [bend left=18, densely dotted, line width=0.4mm, red] (x\i) -- (x\j);
			}	
			
			\foreach \i/\j in {6/11,6/13,6/15}
			{
				\draw  [bend left=18, line width=0.4mm, blue] (x\i) -- (x\j);
			}
			
			\foreach \i/\j in {6/7,6/9}
			{
				\draw  [bend left=18, densely dotted, line width=0.4mm, red] (x\i) -- (x\j);
			}	
			
			\foreach \i/\j in {7/12,7/14,7/0}
			{
				\draw  [bend left=18, line width=0.4mm, blue] (x\i) -- (x\j);
			}
			
			\foreach \i/\j in {7/8,7/10}
			{
				\draw  [bend left=18, densely dotted, line width=0.4mm, red] (x\i) -- (x\j);
			}	
			
			\foreach \i/\j in {8/13,8/15,9/14,9/0,10/15,11/0}
			{
				\draw  [bend left=18, line width=0.4mm, blue] (x\i) -- (x\j);
			}
			
			\foreach \i/\j in {8/9,8/11,9/10,10/11,11/12,12/13,13/14,14/15,15/0,9/12,10/13,11/14,12/15,13/0}
			{
				\draw  [bend left=18, densely dotted, line width=0.4mm, red] (x\i) -- (x\j);
			}		
		\end{tikzpicture}
		\caption{$B_{16;5}$}
		\label{fig:B16;5}
	\end{minipage}
	\begin{minipage}{.45\textwidth}
		\centering
		\begin{tikzpicture}
			[scale=.22]		
			\foreach \i in {1,3,5,7,9}
			{
				\draw[rotate=360-36*(\i)] (0, 12) node[circle, fill=lightgray, draw=black!80, inner sep=0mm, minimum size=3.3mm] (x\i){\scriptsize ${\i}$};
			}
			
			\foreach \i in {0,2,4,6,8}
			{
				\draw[rotate=360-36*(\i)] (0, 12) node[circle, draw=black!80, inner sep=0mm, minimum size=3.3mm] (x\i){\scriptsize ${\i}$};
			}
			
			\foreach \i/\j in {1/4,1/6,1/8}
			{
				\draw  [bend left=18, line width=0.4mm, blue] (x\i) -- (x\j);
			}
			
			\foreach \i/\j in {1/2,1/0}
			{
				\draw  [bend left=18, densely dotted, line width=0.4mm, red] (x\i) -- (x\j);
			}	
			
			\foreach \i/\j in {3/6,3/8,3/0}
			{
				\draw  [bend left=18, line width=0.4mm, blue] (x\i) -- (x\j);
			}
			
			\foreach \i/\j in {3/2,3/4}
			{
				\draw  [bend left=18, densely dotted, line width=0.4mm, red] (x\i) -- (x\j);
			}	
			
			\foreach \i/\j in {5/8,5/0,5/2}
			{
				\draw  [bend left=18, line width=0.4mm, blue] (x\i) -- (x\j);
			}
			
			\foreach \i/\j in {5/4,5/6}
			{
				\draw  [bend left=18, densely dotted, line width=0.4mm, red] (x\i) -- (x\j);
			}	
			
			\foreach \i/\j in {7/0,7/2,7/4}
			{
				\draw  [bend left=18, line width=0.4mm, blue] (x\i) -- (x\j);
			}
			
			\foreach \i/\j in {7/6,7/8}
			{
				\draw  [bend left=18, densely dotted, line width=0.4mm, red] (x\i) -- (x\j);
			}	
			
			\foreach \i/\j in {9/2,9/4,9/6}
			{
				\draw  [bend left=18, line width=0.4mm, blue] (x\i) -- (x\j);
			}
			
			\foreach \i/\j in {9/8,9/0}
			{
				\draw  [bend left=18, densely dotted, line width=0.4mm, red] (x\i) -- (x\j);
			}			
		\end{tikzpicture}
		\caption{$B_{10;3}$}
		\label{fig:B10;3}
	\end{minipage}
\end{figure}

When $p\equiv 2 \; \pmod 4$, and again noting our assumption that $\frac{p}{q}$ is in its simplest form subject to $p$ being even, the signed graph $B_{p;q}$ is already a core with respect to the switching homomorphism. For example, $B_{10;3}$ is depicted in Figure~\ref{fig:B10;3}. In this case then to have a signed circular clique with respect to the edge-sign preserving homomorphism, we must consider DSG$(B_{p;q})$. To be more precise, when $p\equiv 2 \; \pmod 4$, for a signed bipartite graph $(G, \sigma)$ to satisfy that $\chi_c(G, \sigma)\leq \frac{q}{q}$, it is necessary and sufficient that $(G, \sigma)$ admits a (switching) homomorphism to $B_{p;q}$. However, for some choices of $\sigma$, this homomorphism might not be an edge-sign preserving homomorphism and a switching might be necessary. To be sure to have an edge-sign preserving homomorphism then we must consider DSG$(B_{p;q})$. For the example of $p=6$ and $q=2$, which corresponds to circular chromatic number at most $3$, see Figures~\ref{fig:hatK62}, \ref{fig:K62}, \ref{fig:B62} and \ref{fig:DSGB62}. The first one,  the signed graph of Figure~\ref{fig:hatK62} on three vertices, is the signed circular $3$-clique with respect to the switching homomorphism. The second one, the signed graph of Figure~\ref{fig:K62}, which is the Double Switch Graph of the first one, is the signed circular $3$-clique with respect to the edge-sign preserving homomorphism. The third one, the signed graph of Figure~\ref{fig:B62}, also on $6$ vertices, is the signed bipartite circular $3$-clique with respect to the switching homomorphism. Finally the last one, the signed graph of Figure~\ref{fig:DSGB62}, on $12$ vertices, is the Double Switch Graph of the previous one and is the signed bipartite circular $3$-clique with respect to the edge-sign preserving homomorphism.

\begin{figure}[htbp]
	\centering
\renewcommand{\figurename}{Fig}
\captionsetup{justification=centering,margin=9mm}

	\begin{minipage}[t]{.20\textwidth}
		\centering
		\begin{tikzpicture}
			[scale=.3]
			\foreach \i in {0,1,2}
			{
				\draw[rotate=360-120*(\i)] (0, 4.5) node[circle, draw=black!80, inner sep=0mm, minimum size=3.3mm] (x\i){\scriptsize ${\i}$};
			}
			\foreach \i/\j in {1/2,2/0,0/1}
			{
				\draw  [bend left=18, line width=0.4mm, blue] (x\i) -- (x\j);
			}
			
			\foreach \i/\j in {0,1,2}
			{
				\draw[rotate=360-120*(\i)] [bend left=18, densely dotted, line width=0.4mm, red] (x\i) .. controls (3,7) and (-3,7) .. (x\i);
			}	
		
		    	\foreach \i/\j in {3}
		    {
		    	\draw[rotate=360-60*(\i)] [bend left=18, densely dotted, line width=0.4mm, white] (x\i) .. controls (3,7) and (-3,7) .. (x\i);
		    }	
		\end{tikzpicture}
		\caption{$\hat{K}^s_{6;2}$}
		\label{fig:hatK62}
	\end{minipage}\hfil
	\begin{minipage}[t]{.25\textwidth}
		\centering
		\begin{tikzpicture}
			[scale=.34]
			\foreach \i in {0,1,2,3,4,5}
			{
				\draw[rotate=360-60*(\i)] (0, 4.5) node[circle, draw=black!80, inner sep=0mm, minimum size=3.3mm] (x\i){\scriptsize ${\i}$};
			}	
			\foreach \i/\j in {1/2,2/3,3/4,4/5,5/0,0/1}
			{
				\draw  [bend left=18, densely dotted, line width=0.4mm, red] (x\i) -- (x\j);
			}
			\foreach \i/\j in {1/3,3/5,5/1,2/4,4/0,0/2,1/4,2/5,3/0}
			{
				\draw  [bend left=18, line width=0.4mm, blue] (x\i) -- (x\j);
			}
			\foreach \i/\j in {0,1,2,3,4,5}
			{
				\draw[rotate=360-60*(\i)] [bend left=18, densely dotted, line width=0.4mm, red] (x\i) .. controls (3,7) and (-3,7) .. (x\i);
			}	
		\end{tikzpicture}
		\caption{$K^s_{6;2}$}
		\label{fig:K62}
	\end{minipage}\hfil
	\begin{minipage}[t]{.25\textwidth}
		\centering
		\begin{tikzpicture}
			[scale=.36]
			\foreach \i in {1,3,5}
			{
				\draw[rotate=360-60*(\i)] (0, 5) node[circle, draw=black!80, inner sep=0mm, minimum size=3.3mm] (x\i){\scriptsize ${\i}$};
			}	
			
			\foreach \i in {0,2,4}
			{
				\draw[rotate=360-60*(\i)] (0, 5) node[circle, draw=black!80, fill=lightgray, inner sep=0mm, minimum size=3.3mm] (x\i){\scriptsize ${\i}$};
			}	
			\foreach \i/\j in {1/2,2/3,3/4,4/5,5/0,0/1}
			{
				\draw  [bend left=18, densely dotted, line width=0.4mm, red] (x\i) -- (x\j);
			}
			\foreach \i/\j in {1/4,2/5,3/0}
			{
				\draw  [bend left=18, line width=0.4mm, blue] (x\i) -- (x\j);
			}
			\foreach \i/\j in {0,1,2,3,4,5}
			{
				\draw[rotate=360-60*(\i)] [bend left=18, densely dotted, line width=0.4mm, white] (x\i) .. controls (3,7) and (-3,7) .. (x\i);
			}	
		\end{tikzpicture}
		\caption{$B_{6;2}$}
		\label{fig:B62}
	\end{minipage}\hfil
	\begin{minipage}[t]{.30\textwidth}
		\centering
		\begin{tikzpicture}
			[scale=.4]
			\foreach \i in {1,3,5}
			{
				\draw[rotate=360-60*(\i)+9] (0, 5) node[circle, draw=black!80, inner sep=0mm, minimum size=3.3mm] (x\i){\scriptsize ${\i}$};
			}	
			\foreach \i/\j in {1'/1,3'/3,5'/5}
			{
				\draw[rotate=360-60*(\j)-9] (0, 5) node[circle, draw=black!80, inner sep=0mm, minimum size=3.3mm] (x\i){\scriptsize ${\i}$};
			}	
			
			\foreach \i in {0,2,4}
			{
				\draw[rotate=360-60*(\i)+9] (0, 5) node[circle, draw=black!80, fill=lightgray, inner sep=0mm, minimum size=3.3mm] (x\i){\scriptsize ${\i}$};
			}
			\foreach \i/\j in {0'/0,2'/2,4'/4}
			{
				\draw[rotate=360-60*(\j)-9] (0, 5) node[circle, draw=black!80, fill=lightgray, inner sep=0mm, minimum size=3.3mm] (x\i){\scriptsize ${\i}$};
			}

			\foreach \i/\j in {1/2,2/3,3/4,4/5,5/0,0/1}
			{
				\draw  [bend left=18, densely dotted, line width=0.4mm, red] (x\i) -- (x\j);
			}
			\foreach \i/\j in {1/4,2/5,3/0, 1'/4',2'/5',3'/0'}
			{
				\draw  [bend left=18, line width=0.4mm, blue] (x\i) -- (x\j);
			}

			\foreach \i/\j in {1'/2,2'/3,3'/4,4'/5,5'/0,0'/1,1/2',2/3',3/4',4/5',5/0',0/1'}
			{
				\draw  [bend left=18, line width=0.4mm, blue] (x\i) -- (x\j);
			}
			\foreach \i/\j in {1'/4,2'/5,3'/0, 4'/1, 5'/2, 0'/3}
			{
				\draw  [bend left=18, densely dotted, line width=0.4mm, red] (x\i) -- (x\j);
			}
			
			\foreach \i/\j in {0,1,2,3,4,5}
			{
				\draw[rotate=360-60*(\i)] [bend left=18, densely dotted, line width=0.4mm, white] (x\i) .. controls (3,7) and (-3,7) .. (x\i);
			}	
			
		\end{tikzpicture}
		\caption{DSG$(B_{6;2})$}
		\label{fig:DSGB62}
	\end{minipage}
\end{figure}

To distinguish which of the two notions of homomorphisms we are working with, we may define $B^s_{p;q}$ and $\hat{B}^s_{p;q}$ as follows. 

Given a positive even integer $p$ and a positive integer $q$ such that subject to $p$ being even, $\frac{p}{q}$ is in its simplest form and $\frac{p}{q}\geq 2$, we define $B^s_{p;q}$ to be $B_{p;q}$ when $4 \mid p$ and to be DSG$(B_{p;q})$ when $4 \nmid p$. As mentioned before, these signed graphs play the role of signed bipartite circular clique with respect to the edge-sign preserving homomorphism. For the switching homomorphism, we define $\hat{B}^s_{p;q}$ to be $B_{p;q}$ when $4 \nmid p$ and to be the subgraph of $B_{p;q}$ induced on the vertices $\{0, \ldots, \frac{p}{2}-1\}$ when $4 \mid p$. 

We should also note that in defining $K^s_{p;q}$ and $\hat{K}^s_{p;q}$, we did not need to assume $\frac{p}{q}$ is in the simplest form, however, we note that $K^s_{ap;aq}$ and $\hat{K}^s_{ap;aq}$ map, respectively, to $K^s_{p;q}$ and $\hat{K}^s_{p;q}$. By taking such a homomorphism and then taking the pre-image of $B^s_{p;q}$ and $\hat{B}^s_{p;q}$, one may define $B^s_{ap;aq}$ and $\hat{B}^s_{ap;aq}$.

That $B^s_{p;q}$ and $\hat{B}^s_{p;q}$ play the role of circular cliques in the subclass of signed bipartite graphs is the subject of the next theorem. For simplicity, it is stated using $B_{p;q}$ and switching homomorphism but one can easily restated by using $B^s_{p;q}$ or $\hat{B}^s_{p;q}$ and the associated notion of homomorphism.

\begin{theorem}\label{thm:BK}
	Given a signed bipartite graph $(G, \sigma)$ and a rational number $\dfrac{p}{q}$ in $[2,4]$ where $p$ is a positive even integer and subject to this, $\frac{p}{q}$ is in its simplest form, we have 
	$$\chi_c(G, \sigma)\leq \dfrac{p}{q} \text{ if and only if } (G, \sigma)\to B_{p;q}.$$
\end{theorem}

\begin{proof}
Let $(G, \sigma)$ be a signed bipartite graph. One direction is quite trivial. As $B_{p;q}$ is a subgraph of $K^s_{p;q}$, $\chi_c(B_{p;q})\leq \frac{p}{q}$. If $(G, \sigma)\to B_{p;q}$, then, by Lemma~\ref{lem:No-HomByCircularChromatic}, we have $\chi_c(G, \sigma)\leq \frac{p}{q}$. 

It remains to show that if $\chi_c(G, \sigma)\leq \frac{p}{q}$, then $(G, \sigma)\to B_{p;q}.$ Since $B_{p;q}$ behaves differently depending on whether $p$ divides $4$ or not, we divide the proof into two cases based on this criteria: (1) $p=4k$. (2) $p=4k+2$. We note that in the first case, $q$ must be an odd number.

\medskip
\noindent
{\bf Case 1} {\em $p=4k$.}
\medskip

As $\frac{p}{q}\geq 2$, we know $q$ is an odd number smaller than $2k-1$. Let $(X, Y)$ be the bipartition of $B_{4k;q}$ and let $(A, B)$ be the bipartition of $(G, \sigma)$. Since $\chi_c(G, \sigma)\leq \frac{4k}{q}$, there is an edge-sign preserving homomorphism of $(G, \sigma)$ to $K^s_{4k;q}$. Let $\varphi$ be such a homomorphism. Our goal is to modify $\varphi$, if needed, so that we obtain a mapping of $(G, \sigma)$ to $B_{4k;q}$. This would of course be based on the bipartition of $G$. One such modification is given as follows:

$$\phi(u)=
\begin{cases}
     \varphi(u)+1 & \text{ either $u\in A$ and $\varphi(u)\in Y$ or $u\in B$ and $ \varphi(u) \in X$, } \\ 
	 \varphi(u) & \text{ otherwise. } \\
\end{cases}
$$

Intuitively, we aim at modifying the mapping such that the vertices in the part $A$ of $G$ are mapped to the vertices in the part $X$ of $B_{4k;q}$ and the vertices in the part $B$ are mapped to the vertices in the part $Y$. In defining $\phi$, for vertices of $G$ satisfying these conditions under the mapping $\varphi$, we give a same image under $\phi$. If this condition is not met, then we shift the image by $1$ in the clockwise direction of the circle. What remains is to show that $\phi$ is also an edge-sign preserving homomorphism of $(G, \sigma)$ to $K^s_{4k;q}$. Then it would naturally be a homomorphism of $(G,\sigma)$ to $B_{4k;q}$ as well. 

Given an edge $e=uv$ of $G$, if both $\phi(u)=\varphi(u)$ and $\phi(v)=\varphi(v)$ hold, then $e$ is already mapped to an edge of a same sign under $\varphi$ and nothing left to show. If $\phi(u)=\varphi(u)+1$ and $\phi(v)=\varphi(v)+1$, then the claim follows from the circular structure of $K^s_{p;q}$, that is, if there is an edge $ij$ of sign $\eta$ in $B_{4k;q}$, then there is also an $(i+1)(j+1)$ (additions done modulo $4k$) edge of sign $\eta$. It remains to consider the case that only one endpoint of $e=uv$ has been shifted. By the symmetry, we may assume $\phi(u)=\varphi(u)$ and $\phi(v)=\varphi(v)+1$. Moreover, noting that $u$ and $v$ must be in different parts of the bipartite graph $G$, and again by the symmetries, we assume $u\in A$ and $v\in B$ with $\varphi(u), \varphi(v) \in X$. Hence, by our assumption, $\phi(u)=\varphi(u)$ and $\phi(v)=\varphi(v)+1 \in Y$. Depending on the signature of $e$, we consider two cases. 

If $e$ is a positive edge, then $\varphi(u)\varphi(v)$ is a positive edge of $K^s_{4k;q}$. Thus $q \leq |\varphi(u)-\varphi(v)| \leq p-q=4k-q$. Observe that, as $\varphi(u)$ and $\varphi(v)$ are both in $X$, they have a same parity, and thus $|\varphi(u)-\varphi(v)|$ is an even number. However, since $q$ is an odd number, both sides of the inequality (i.e., $q$ and $4k-q$) are odd numbers and, therefore, equality cannot hold there. It is implied that if we change (only) one of $\varphi(u)$ and $\varphi(v)$ by a value of at most $1$, then the inequality would still hold. Thus $\phi(u)\phi(v)$ is a positive edge of $K^s_{4k;q}$.

If $e$ is a negative edges, then (only) one of the following must hold: either $|\varphi(u)-\varphi(v)| \leq \frac{p}{2}-q=2k-q$ or $|\varphi(u)-\varphi(v)| \geq \frac{p}{2}+q=2k+q$. As in the previous case, we conclude that $|\varphi(u)-\varphi(v)|$ is an even number. However, $\frac{p}{2}=2k$ is an even number while $q$ must be an odd number. Thus both of $2k-q$ and $2k+q$ are odd numbers and once again the equality cannot hold. Therefore, after shifting only one of the values of $\varphi(u),\varphi(v)$ by $1$, the corresponding inequality holds with respect to the new function which is $\phi$, that is to say, either $|\phi(u)-\phi(v)| \leq 2k-q$ or $|\phi(u)-\phi(v)| \geq 2k+q$. Hence $e$ is mapped to a negative edge $\phi(u)\phi(v)$ of $K^s_{4k;q}$.

\medskip
\noindent
{\bf Case 2} {\em $p=4k+2$.}
\medskip

Notice that in this case, $\frac{p}{2}=2k+1$ is an odd number. Let $(X, Y)$ be the bipartition of $B_{4k+2;q}$ and let $(A, B)$ be a bipartition of $(G, \sigma)$. 

Since $\chi_c(G, \sigma)\leq \frac{4k+2}{q}$, there exists an edge-sign preserving homomorphism of $(G, \sigma)$ to $K^s_{4k+2;q}$, say $\varphi$. Our goal is to modify $\varphi$ to obtain a (switching) homomorphism of $(G, \sigma)$ to $B_{4k+2;q}$. This would be based on the bipartition of $G$. Intuitively, we want a mapping that maps vertices in $A$ to $X$ and those in $B$ to $Y$. We observe that for each pair of antipodal vertices of $K^s_{4k+2;q}$, one is in $X$ and the other is in $Y$. Thus in the mapping $\varphi$, if one vertex is not mapped to the correct part, then we first apply a switching at that vertex and then map it to the antipodal of the original image. This is formalized as follows.

$$\phi(u)=
\begin{cases}
	\overbar{\varphi(u)} \text{ (switching at $u$)} & \text{ either $u\in A$ and $\varphi(u)\in Y$ or $u\in B$ and $ \varphi(u) \in X$, } \\ 
	\varphi(u) & \text{ otherwise. } \\
\end{cases}
$$

What remains is to show that $\phi$ is a (switching) homomorphism of $(G, \sigma)$ to $K^s_{4k+2;q}$. Then it would naturally be a homomorphism of $(G,\sigma)$ to $B_{4k+2;q}$ as well. 

Given an edge $e=uv$ of $G$, if both $\phi(u)=\varphi(u)$ and $\phi(v)=\varphi(v)$ hold, then it follows easily that $\phi(u)\phi(v)$ is the required edge. If $\phi(u)=\overbar{\varphi(u)}$ and $\phi(v)=\overbar{\varphi(u)}$, then we switch at both of vertices $u$ and $v$. Thus the sign of $uv$ does not change. Moreover, vertices $i$ and $j$ are connected by an edge of sign $\eta$ in $K^s_{4k+2;q}$, then their antipodals are also connected by an edge of the same sign. Therefore, $\phi(u)\phi(v)$ is an edge of $K^s_{4k+2;q}$ with the same sign as $\varphi(u)\varphi(v)$ and thus as $uv$. The final case is that only one endpoint of $e=uv$ has been switched and mapped to the antipodal. By the symmetries, we may assume that we switch at $v$ and $\phi(v)=\overbar{\varphi(v)}$. Moreover, noting that $u$ and $v$ must be in different parts of the bipartite graph $G$, and again by the symmetries, we assume $u\in A$ and $v\in B$ with $\varphi(u), \varphi(v) \in X$. Hence, by our assumption, $\phi(u)=\varphi(u)$ and $\phi(v)=\overbar{\varphi(v)} \in Y$. Depending on the sign of $e$, we consider two cases. 

If $e$ is a positive edge, then $\varphi(u)\varphi(v)$ is a positive edge of $K^s_{4k+2;q}$. Thus $q \leq |\varphi(u)-\varphi(v)| \leq p-q=4k+2-q$. As $|\varphi(v)-\overbar{\varphi(v)}|=\frac{p}{2}=2k+1$, we have $|\varphi(u)-\overbar{\varphi(v)}|\leq 2k+1-q$ or $|\varphi(u)-\overbar{\varphi(v)}|\geq 2k+1+q$. Note that now $uv$ is a negative edge of $(G, \sigma')$ where $\sigma'$ is obtained from $\sigma$ by switching at $v$. Since $\varphi(u)\overbar{\varphi(v)}$ satisfies the condition for being a negative edge of $K^s_{4k+2;q}$, $\phi(u)\phi(v)$ is a negative edge that we required.

If $e$ is a negative edges, then (only) one of the following must hold: either $|\varphi(u)-\varphi(v)| \leq \frac{p}{2}-q=2k+1-q$ or $|\varphi(u)-\varphi(v)| \geq \frac{p}{2}+q=2k+1+q$. As in the previous case, we have that $|\varphi(v)-\overbar{\varphi(v)}|=\frac{p}{2}=2k+1$. Thus $q \leq |\varphi(u)-\overbar{\varphi(v)}| \leq 4k+2-q$. Switching at $v$ makes $uv$ become a positive edge. Now $\varphi(u)\overbar{\varphi(v)}$ satisfies the condition for being a positive edge of $K^s_{4k+2;q}$, in other words, $\phi(u)\phi(v)$ is a positive edge. Therefore, we verify that $\phi$ is a (switching) homomorphism of $(G, \sigma)$ to $K^s_{4k+2;q}$ and thus also to its signed bipartite subgraph $B_{4k+2;q}$.
\end{proof}

We note that the assumption $\frac{p}{q}\leq 4$ is not used explicitly in the proof. If $\frac{p}{q}> 4$, then the signed bipartite graph induced by odd versus even vertices will contain a digon which admits a homomorphism from any signed bipartite graph and provides the upper bound of $4$ for the circular chromatic number of this class of signed graphs. That leaves us with circular $3$-coloring as a special case. In this case, by switching at all vertices of one part of $B_{6;2}$, we get a signed graph on $K_{3,3}$ where all negative edges form a perfect matching. Thus we denote it by $(K_{3,3},M)$. Hence, as a special case we have:

\begin{corollary}
	Given a signed bipartite graph $(G, \sigma)$, $$\chi_c(G, \sigma)\leq 3 \text{ if and only if } (G, \sigma)\to (K_{3,3}, M).$$
\end{corollary}

Another special case is when $p=4k$ and $q=2k-1$. In this case, one may observe that the (switching) core of $B_{4k;2k-1}$ (on $2k$ vertices) is switching equivalent to the negative cycle $C_{-2k}$. Hence, we have the following corollary.

\begin{corollary}
	Given a signed bipartite graph $(G, \sigma)$,  
	$$\chi_c(G, \sigma)\leq \dfrac{4k}{2k-1} \text{ if and only if } (G, \sigma)\to C_{\!\scriptscriptstyle -2k}.$$
\end{corollary}

This helps to fill the parity gap in some study of circular coloring of graphs where homomorphism to odd cycle $C_{2k+1}$ is known to be equivalent to a circular $\frac{4k+2}{2k}$-coloring. A uniform presentation of the two is as follows.

\begin{theorem}
Given a positive integer $\ell$, $\ell\geq 2$, and a signed graph $(G, \sigma)$ satisfying $g_{ij}(G, \sigma) \geq g_{ij}(C_{-\ell})$ for each $ij\in \mathbb{Z}_2^2$, we have
 	$$\chi_c(G, -\sigma)\leq \dfrac{2\ell}{\ell-1} \text{ if and only if } (G, \sigma)\to C_{\!\scriptscriptstyle -\ell}.$$
\end{theorem}

\subsection{Circular coloring and subdivision}

A classic relation between the chromatic number of a graph and homomorphism from a certain subdivisions of it to the odd cycle is extended, in \cite{NPW20+}, to a relation between the circular chromatic number of signed graphs and homomorphism of its subdivision to negative cycles. Here we present a slightly stronger version and then use it to build examples in the next sections. 

\begin{definition}
Given a signed graph $(G, \sigma)$ and a positive integer $\ell$, we define $T^*_{\ell}(G, \sigma)$ to be the signed graph obtained from $(G, \sigma)$ by replacing each edge $e$ with a path $P_{\ell}$ of length $\ell$ where internal vertices of the path are disjoint and assigning a signature satisfying that $P_{\ell}$ contains an odd number of positive edges if $e$ is a positive edge and $P_{\ell}$ contains an even number of positive edges if $e$ is a negative edge.
\end{definition}

We note that there are many choices for the signature in defining $T^*_{\ell}(G, \sigma)$, but, as all such choices are switching equivalent, one may take any. The relation between the circular chromatic number of $(G, \sigma)$ and $T^*_{\ell}(G, \sigma)$ follows from two lemmas based on the following notation of indicator. 

Given a signed graph $I$ with two specific vertices $u$ and $v$, we refer to $\mathcal{I}=(I,u,v)$ as an indicator.
Given an indicator $\mathcal{I}$ and a real number $r$, $r\geq 2$, with $[0,r)$ viewed as the circle of circumference $r$, we define $Z(\mathcal{I})$ to be the set of possible choices for $u$ in a circular $r$-coloring of $\mathcal{I}$ where $v$ is colored by $0$.
  
Given two indicators $\mathcal{I}_+=(I_1,u_1,v_2)$ and $\mathcal{I}_-=(I_2,u_2,v_2)$, for each signed graph $\Omega$, we define $\Omega(\mathcal{I}_+, \mathcal{I}_-)$ to be the signed graph obtained from $\Omega$ by replacing each positive edge $xy$ with a distinct copy of $\mathcal{I}_+$ where $x$ is identified with $u_1$ and $y$ with $u_2$ and similarly replacing each negative edge with $\mathcal{I}_-$. For some indicators, the circular chromatic number of $\Omega(\mathcal{I}_+, \mathcal{I}_-)$ could be determined by $\chi_c(\Omega)$.

\begin{lemma}{\rm \cite{NWZ21}}\label{lem:indicator}
    	Assume $\mathcal{I}_+$ and $\mathcal{I}_-$ are two signed indicators, $r\geq 2$ is a real number such that $Z(\mathcal{I}_+)=[t, \frac{r}{2}] \text{ and } Z(\mathcal{I}_-)=[0, \frac{r}{2}-t]$ for some $0<t<\frac{r}{2}$. Then for any signed graph $\Omega$, we have $$\chi_c(\Omega(\mathcal{I}_+, \mathcal{I}_-))=t\chi_c(\Omega).$$
\end{lemma}

We denote a path of length $\ell$ which contains an odd number of positive edges by $P^{o}_{\ell}$ and a path of length $\ell$ which contains an even number of positive edges by $P^{e}_{\ell}$.
The special choice for the indicators are $P^{o}_{\ell}$ and $P^{e}_{\ell}$, with the two endpoints as special vertices. The range of possible choices for the ends in the circular colorings of these paths is computed in \cite{PZ21+}.

\begin{lemma}{\rm \cite{PZ21+}}\label{lem:Path-Indicator}
		Given an integer $\ell\geq 1$ and a real number $r< \frac{2\ell}{\ell-1}$, $$Z(P^{e}_{\ell})=[0, \ell\frac{r}{2}-\ell] \text{ and } Z(P^{o}_{\ell})=[\ell-(\ell-1)\frac{r}{2}, \frac{r}{2}].$$
\end{lemma}
   
Combining these two lemmas, where we take $\mathcal{I}_{+}=P^{o}_{\ell}, \mathcal{I}_{-}=P^{e}_{\ell}$ and $t=\ell-(\ell-1)\frac{r}{2}$, we have the following.

\begin{lemma}\label{lem:X_c(T_l)}
For any signed graph $\Omega$, $$\chi_c(T^*_\ell(\Omega))=\frac{2\ell\chi_c(\Omega)}{(\ell-1)\chi_c(\Omega)+2}.$$
\end{lemma}

A few comments are to be mentioned here. 

The first is to note that for each positive integer $\ell$, and by considering a graph $G$ as a signed graph where all edges are regarded positive, the 4-colorability of $G$ is equivalent to proving that $\chi_c(T^*_\ell(G))\leq \frac{8\ell}{4\ell-2}$. Furthermore, noting that subdivision preserves the planarity, for each choices of $\ell$, we have a reformulation of the $4$-Color Theorem. For each such $\ell$ then one line of study is to introduce an interesting classes of signed graphs that includes $T^*_\ell(G)$ for all planar graphs $G$ and admits the same upper bound for the circular chromatic number.

A second note here is that for even values of $\ell$, $T^*_\ell(\Omega)$ is a signed bipartite graph and thus the subject of the main study in this work. 

And the last note is that the $S(G)$ construction mentioned before is a special case of this indicator construction. Given a graph $G$, one may first build a signed graph $\tilde{G}$ by replacing each edge with a digon. Then $T^*_{2}(\tilde{G})$ is the same as $S(G)$.

\section{Coloring planar signed graphs}

We have already noted that, via constructions such as $S(G)$, most homomorphisms and coloring question can be restated in the language of homomorphisms of singed bipartite graphs. Here we have a look at what this means to the coloring of planar graphs and restate some famous theorems such as the $4$-Color Theorem and Gr\"otzsch's theorem. 

As a direct corollary of Theorem~\ref{thm:S(G)->S(H)}, we have the following reformulation of $4$-Color Theorem.

\begin{theorem}\label{thm:4CT-restated2}{\rm [4CT restated]}
	For any planar graph $G$, we have $S(G)\to S(K_4)$.
\end{theorem}

Noting that the 4-Color Theorem is equivalent to bounding the circular chromatic number of all planar graphs by 4, and applying Theorem~\ref{thm:X(G)-->X(S(G))}, another restatement of 4-Color Theorem is as follows.

\begin{theorem}\label{thm:4CT-restated3}{\rm [4CT restated]}
	For any planar graph $G$, we have $\chi_c(S(G))\leq \frac{16}{5}$.
\end{theorem}

Since every $S(G)$ is a signed bipartite graph, the claim of this theorem is equivalent to the existence of an edge-sign preserving mapping from $S(G)$ to $B^s_{16;5}$. Note that $\hat{B}^s_{16;5}$, the switching core of $B^s_{16;5}$, is a signed graph on $K_{4,4}$. With one random choice of a signature, (among all equivalent signatures), this core is presented in Figure~\ref{fig:hatB16;5}. We recall that the 4-Color Theorem is also restated in Theorem~\ref{thm:Restated4CT} in the form of mapping $S(G)$, for every planar $G$, to $(K_{4,4}, M)$.

Let $\mathcal{P}'$ be the class of all simple planar graphs and let $S(\mathcal{P}')=\{S(G) \mid G\in \mathcal{P}'\}$. 
Then, by the discussion above, the 4-Color Theorem is equivalent to bounding the class $S(\mathcal{P}')$ by either of the signed bipartite graphs of Figures~\ref{fig:S(K4)}, \ref{fig:K44M}, or \ref{fig:hatB16;5}. 
One may observe that $S(K_4)$ admits a homomorphism to both $(K_{4,4}, M)$ and $\hat{B}^s_{16;5}$ but that $(K_{4,4}, M)$ and $\hat{B}^s_{16;5}$ are homomorphically incomparable. The latter is a consequence of the following two facts: 1. Any pair of nonadjacent vertices in $(K_{4,4}, M)$ or in $\hat{B}^s_{16;5}$ belongs to a negative $4$-cycle which means identifying them would result in a digon, 2. the two signed graphs are not switching isomorphic, for example, $\chi_c(K_{4,4}, M)=4$ and $\chi_c(\hat{B}^s_{16;5})=\frac{16}{5}$.

\begin{figure}[htbp]
	\centering
\begin{minipage}[t]{.3\textwidth}
	\centering
	\begin{tikzpicture}
	[scale=.32]
	\draw (0,0) node[circle, draw=black!80, inner sep=0mm, minimum size=3.3mm] (w){\scriptsize $w$};
	\draw[rotate=0] (0,6) node[circle, draw=black!80, inner sep=0mm, minimum size=3.3mm] (x){\scriptsize $x$};
	\draw[rotate=120] (0,6) node[circle, draw=black!80, inner sep=0mm, minimum size=3.3mm] (y){\scriptsize $y$};
	\draw[rotate=240] (0,6) node[circle, draw=black!80, inner sep=0mm, minimum size=3.3mm] (z){\scriptsize $z$};
	
	\draw[rotate=60] (0,6) node[circle, draw=black!80, inner sep=0mm, minimum size=1.7mm] (v_{xy}){};
	\draw[rotate=180] (0,6)  node[circle, draw=black!80, inner sep=0mm, minimum size=1.7mm] (v_{yz}){};
	\draw[rotate=300] (0,6) node[circle, draw=black!80, inner sep=0mm, minimum size=1.7mm] (v_{xz}){};

	\draw [rotate=60] (0,4.5) node[circle, draw=black!80, inner sep=0mm, minimum size=1.7mm] (u_{xy}){};
	\draw [rotate=180] (0,4.5) node[circle, draw=black!80, inner sep=0mm, minimum size=1.7mm] (u_{yz}){};
	\draw [rotate=300] (0,4.5) node[circle, draw=black!80, inner sep=0mm, minimum size=1.7mm] (u_{xz}){};

	\draw [rotate=15] (0,3) node[circle, draw=black!80, inner sep=0mm, minimum size=1.7mm] (u_{xw}){};
	\draw [rotate=135] (0,3) node[circle, draw=black!80, inner sep=0mm, minimum size=1.7mm] (u_{yw}){};	
	\draw [rotate=255] (0,3) node[circle, draw=black!80, inner sep=0mm, minimum size=1.7mm] (u_{zw}){};	
	
	\draw [rotate=-15] (0,3) node[circle, draw=black!80, inner sep=0mm, minimum size=1.7mm] (v_{xw}){};
	\draw [rotate=105] (0,3) node[circle, draw=black!80, inner sep=0mm, minimum size=1.7mm] (v_{yw}){};		
	\draw [rotate=225] (0,3) node[circle, draw=black!80, inner sep=0mm, minimum size=1.7mm] (v_{zw}){};	
	
	\draw [line width=0.4mm, blue] (x) -- (v_{xy});
	\draw [line width=0.4mm, blue] (x) -- (u_{xy});
	\draw [line width=0.4mm, blue] (x) -- (u_{xz});
	\draw [line width=0.4mm, blue] (x) -- (u_{xw});
	\draw [line width=0.4mm, blue] (x) -- (v_{xw});
	\draw [densely dotted, line width=0.4mm, red] (x) -- (v_{xz});
	\draw [line width=0.4mm, blue] (y) -- (v_{yz});
	\draw [line width=0.4mm, blue] (y) -- (u_{xy});
	\draw [line width=0.4mm, blue] (y) -- (u_{yz});
	\draw [line width=0.4mm, blue] (y) -- (u_{yw});
	\draw [line width=0.4mm, blue] (y) -- (v_{yw});
	\draw [densely dotted, line width=0.4mm, red] (y) -- (v_{xy});
	\draw [line width=0.4mm, blue] (z) -- (u_{yz});
	\draw [line width=0.4mm, blue] (z) -- (v_{xz});
	\draw [line width=0.4mm, blue] (z) -- (u_{xz});
	\draw [line width=0.4mm, blue] (z) -- (u_{zw});
	\draw [line width=0.4mm, blue] (z) -- (v_{zw});
	\draw [densely dotted, line width=0.4mm, red] (z) -- (v_{yz});
	\draw [line width=0.4mm, blue] (w) -- (u_{xw});
	\draw [line width=0.4mm, blue] (w) -- (u_{yw});
	\draw [line width=0.4mm, blue] (w) -- (u_{zw});
	\draw [densely dotted, line width=0.4mm, red] (w) -- (v_{xw});
	\draw [densely dotted, line width=0.4mm, red] (w) -- (v_{yw});
	\draw [densely dotted, line width=0.4mm, red] (w) -- (v_{zw});		
\end{tikzpicture}
\caption{$S(K_4)$}
\label{fig:S(K4)}
\end{minipage}\hfil
\begin{minipage}[t]{.3\textwidth}
	\centering
	\begin{tikzpicture}
		[scale=.2]
		\foreach \i/\j in {1/1,2/3,3/5,4/7}
		{
			\draw(6*\i, 0) node[circle, draw=black!80, inner sep=0mm, minimum size=3.3mm] (x\j){\scriptsize ${\j}$};
		}
		
		\foreach \i/\j in {1/2,2/4,3/6,4/8}
		{
			\draw(6*\i, -12) node[circle, draw=black!80, inner sep=0mm, minimum size=3.3mm] (x\j){\scriptsize ${\j}$};
		}
		
		\foreach \i/\j in {1/4,1/6,1/8,3/2,3/6,3/8,5/2,5/4,5/8,7/2,7/4,7/6}
		{
			\draw  [bend left=18, line width=0.4mm, blue] (x\i) -- (x\j);
		}
		
		\foreach \i/\j in {1/2,3/4,5/6,7/8}
		{
			\draw  [bend left=18, densely dotted, line width=0.4mm, red] (x\i) -- (x\j);
		}	
\draw (10,-16) node[circle, draw=white!80, inner sep=0mm, minimum size=3.3mm] (FAKE){};
	\end{tikzpicture}
	\caption{$(K_{4,4}, M)$}
	\label{fig:K44M}
\end{minipage}\hfil
\begin{minipage}[t]{.3\textwidth}
	\centering
	\begin{tikzpicture}
		[scale=.2]
		\foreach \i/\j in {1/1,2/3,3/7,4/5}
		{
			\draw(6*\i, 0) node[circle, draw=black!80, inner sep=0mm, minimum size=3.3mm] (x\j){\scriptsize ${\j}$};
		}
		
		\foreach \i/\j in {1/2,2/4,3/6,4/8}
		{
			\draw(6*\i, -12) node[circle, draw=black!80, inner sep=0mm, minimum size=3.3mm] (x\j){\scriptsize ${\j}$};
		}
		
		\foreach \i/\j in {1/4,1/6,3/2,3/6,3/8,5/2,5/6,7/2,7/4,7/8}
		{
			\draw  [bend left=18, line width=0.4mm, blue] (x\i) -- (x\j);
		}
		
		\foreach \i/\j in {1/2,1/8,6/7,3/4,4/5,5/8}
		{
			\draw  [bend left=18, densely dotted, line width=0.4mm, red] (x\i) -- (x\j);
		}	
\draw (10,-16) node[circle, draw=white!80, inner sep=0mm, minimum size=3.3mm] (FAKE){};
	\end{tikzpicture}
	\caption{$\hat{B}^s_{16;5}$}
	\label{fig:hatB16;5}
	\end{minipage}
\end{figure}

As $S(K_4)$ itself is in the family $S(\mathcal{P}')$, one does not expect much room to strengthen the result regarding this one beyond expected strengthening of the 4-Color Theorem. For example, the statement holds when $G$ is $K_5$-minor-free, and is expected to hold if $(G,-)$ has no $(K_5,-)$-minor. 

Since $S(K_4)$ admits a homomorphism to each of $(K_{4,4},M)$ and $\hat{B}^s_{16;5}$, it would not be a surprise if a stronger statement can be proved regarding these two targets. Indeed that is the case for $(K_{4,4}, M)$: it bounds the class of all signed bipartite planar simple graphs \cite{NRS13}. As the limit of the circular chromatic number of signed bipartite planar simple graphs is $4$ (see \cite{NWZ21} and  \cite{KNNW21+}), this cannot be the case for $\hat{B}^s_{16;5}$. Thus it remains an open question to bound a larger class of signed bipartite planar graphs with $\hat{B}^s_{16;5}$. 

It is shown in \cite{NPW20+} that every signed bipartite planar graph of negative girth at least $8$ maps to $C_{-4}$ and that this girth condition cannot be improved to $6$. We observe that $C_{-4}$ is a subgraph of each of the three homomorphism targets of this discussion.

Another common subgraph of $(K_{4,4},M)$ and $\hat{B}^s_{16;5}$ which is of high interest for this discussion is $(K_{3,3}, M)$. A restatement of the Gr\"otzsch theorem is the following.

\begin{theorem}{\rm [Gr\"otzsch theorem restated]}
For any triangle-free planar graph $G$ (with no loop), we have $\chi_c(S(G))\leq 3$.
\end{theorem}

In the next section, we prove Theorem~\ref{thm:main} which may be viewed as a parallel theorem to Gr\"otzsch's theorem. In Section~\ref{sec:Conclusion} then we propose a question as potentially common strengthening of the two theorems.

\subsection{Bounding the circular chromatic number by 3}\label{sec:CircularChromaticNumber_3}

For a class $\mathcal{C}$ of signed graphs, we define $\chi_c(\mathcal{C})=sup\{\chi_c(G, \sigma): (G, \sigma)\in \mathcal{C}\}$. 

For a given integer $k$, let $\mathcal{P}^*_{k}$ be the class of singed planar graphs $(G, \sigma)$ such that the signed graph $(G, -\sigma)$ satisfies the following conditions: for each $ij\in \mathbb{Z}_2^2$, we have $g_{ij}(G, -\sigma)\geq g_{ij}(C_{-k})$. Thus for an odd integer $k$, and after suitable switchings, $\mathcal{P}^*_{k}$ consists of all planar graphs of odd girth at least $k$ with all edges being assigned positive signs. For an even value of $k$, the class $\mathcal{P}^*_{k}$ consists of all signed planar bipartite graphs of negative girth at least $k$. 

A main question then is to find $\chi_c(\mathcal{P}^*_{k})$ for each $k$. For $k=3$ and $4$ both answers are $4$, the first by the 4-Color Theorem, the second by the observation that $4$ is the upper bound for the class of signed bipartite simple graphs and a construction given in \cite{NWZ21} showing that $4$ cannot be improved (see also \cite{KNNW21+}).
For $k=5$, we have the Gr\"otzsch theorem, that gives upper bound of $3$ which is also shown to be the optimal value. For $k=4p+1$, this question is the subject of widely studied Jaeger-Zhang conjecture. And for other values of $k$, similar conjectures are proposed. Here, addressing the case $k=6$ we prove the followings.

\begin{theorem}\label{thm:BoundingXc(P6)}
We have $\dfrac{14}{5}\leq \chi_c(\mathcal{P}^*_{6})\leq 3$.
\end{theorem} 

The proof of the upper bound is based on the following theorem which is implied by combining results from \cite{DKK16} and \cite{NRS13}, noting that the $4$-Color Theorem is used in proving this claim.

\begin{theorem}\label{thm:packing}
Given a signed bipartite planar graph $(G,\sigma)$ of negative girth at least $6$, one can find six disjoint subsets of edges, $E_1, E_2, \ldots, E_6$, such that each of the signed graphs $(G, \sigma_i)$, $i\in \{1,2,\ldots, 6\}$, where $E_i$ is the set of negative edges of $(G, \sigma_i)$, is switching equivalent to $(G, \sigma)$.
\end{theorem}

In other words, the signature packing number of $(G, \sigma)$ is at least 6 (see \cite{NY21+} for more details).
\bigskip

\begin{proof*}\ref{thm:BoundingXc(P6)}.
Let $(G, \sigma)$ be a signed bipartite planar graph of negative girth at least $6$
with a bipartition $(A, B)$. By Theorem~\ref{thm:packing}, there are disjoint subsets $E_1, E_2, \ldots, E_{6}$ of edges of $G$ such that for each $i\in [6]$, the signature $\sigma_i$, whose negative edges are $E_i$, is equivalent to $\sigma$.

We consider the signed graph $(G, \sigma_1)$ where the set of negative edges is $E_1$. Let $G'$ be the graph obtained from $G$ by contracting all the edges in $E_1$. In this notion of contracting, we delete the contracted edge (those in $E_1$) but all other edges remain. Thus in theory we may have loops and parallel edge in the resulting graph. However, we show next that not only $G'$ has no loop, it has no triangle either. In other words, we claim that every odd cycle of $G'$ is of length at least $5$. 

To see this, let $C'$ be an odd cycle of $G'$. This cycle is obtained from a cycle $C$ of $G$ by contracting some edges (of $E_1$).  As $G$ is bipartite, $C$ must be of even length. Thus the number of the contracted edges is odd. Therefore, $C$ is a negative cycle in the signed graph $(G,\sigma_1)$. As all the $(G, \sigma_i)$, $i=1,2,\ldots, 6$, are equivalent, $C$ is negative in all of them which means it has an odd number of edges from each of $E_i$'s. As these sets are disjoint, and as for $i=2, 3, \ldots 6$, they still present in $G'$, the cycle $C'$ has an odd number of edges from each $E_i$, $i=2, 3, \ldots 6$. In particular, that is at least one edge from each, and noting again that they are disjoint sets, we conclude that $C'$ is of length at least $5$.

Having shown that $G'$ is a triangle-free planar graph with no loop (might have parallel edges), we may apply the Gr\"otzsch theorem to obtain a $3$-coloring $\varphi: V(G')\to \{1,2,3\}$ of $G'$. Let $(X, Y)$ be the bipartition of $(K_{3,3}, M)$. Label the vertices $X=\{x_1,x_2,x_3\}$ and $Y=\{y_1,y_2,y_3\}$ such that $\{x_1y_1,x_2y_2,x_3y_3\}$ is the set of negative edges. 

We define the mapping $\psi$ of $(G,\sigma_1)$ to $(K_{3,3}, M)$ as follows:
\[
\psi(u) = \begin{cases} x_i, &\text{ if $u\in A$ and $\varphi(u)=i$} \cr
	y_i, &\text{ if $u \in B$ and $\varphi(u)=i$}. \cr
\end{cases}
\] 

It remains to show that $\psi$ is an edge-sign preserving mapping of $(G,\sigma_1)$ to $(K_{3,3}, M)$.
For any positive edge $uv$ of $(G, \sigma_1)$, without loss of generality, we may assume that $u\in A, v\in B$ and that $\varphi(u)=i$, $\varphi(v)=j$. Noting that $uv$ is also an edge of $G'$, as $\varphi$ is a proper $3$-coloring, we have that $i \neq j$. Thus $\psi(u)\psi(v)=x_ix_j$ is a positive edge in $(K_{3,3}, M)$. For any negative edge $uv$ of $(G, \sigma_1)$, without loss of generality, assume $u\in A, v\in B$. As $uv\in E_1$ is contracted to a vertex to obtain $G'$, $\varphi(u)=\varphi(v)=i$. So $\psi(u)\psi(v)=x_iy_i$ is a negative edge. Hence, $\psi$ is an edge-sign preserving homomorphism of $(G, \sigma_1)$ to $(K_{3,3, M})$.

This complete the proof of the upper bound. For the lower bound, the best is to give an example. An example of a simple planar graph $(G, \sigma)$ satisfying $\chi_c(G, \sigma)=\frac{14}{3}$ is given in \cite{NWZ21}. Then it follows from Lemma~\ref{lem:X_c(T_l)} that $\chi_c (T^*_2(G, \sigma))=\frac{14}{5}$. It is easily observed that, since $(G, \sigma)$ is a signed simple planar graph, $T^*_2(G, \sigma)$ has (negative) girth at least $6$ and obviously it is a signed bipartite graph.
\end{proof*}

\subsection{Bounds based on girth} 

We will denote the class of all signed planar graphs with $\mathcal{P}$, where we will allow loops and multi-edges. The subclass of $\mathcal{P}$ where the shortest cycle of each member is at least $k$ will be denoted by $\mathcal{P}_k$. Thus, in particular, $\mathcal{P}_2$ is the class of all loop-free signed planar graphs and $\mathcal{P}_3$ is the class of all signed planar simple graphs.

We note that $\mathcal{P}^*_k$ is not a subclass of $\mathcal{P}_k$ as signed graphs in $\mathcal{P}^*_k$ may have positive even cycles of any length. However, it is expected that the circular chromatic number of $\mathcal{P}^*_k$ is determined by the subclass $\mathcal{P}^*_k \cap \mathcal{P}_k$.

The questions of determining $\chi_c(\mathcal{P}^*_k)$ is closely related to some of the most well known theorems and conjectures in the theory of graph coloring, such the $4$-Color theorem, Gr\"otzsch's theorem and Jaeger-Zhang conjecture. This also leads to the importance of the question of determining $\chi_c(\mathcal{P}_k)$. In the table below we summarize the best known results for these questions for some value of $k$.

\newpage
\begin{center}
	{\bf  Circular chromatic number of $\mathcal{P}^*_k$ and $\mathcal{P}_k$}
	\fontsize{10}{10}
	\begin{tabular}{ | l | l | l | l | l | }
		\hline
		$k$ & Bounds on $\chi_c(\mathcal{P}^*_k)$ & Reference & Bounds on $\chi_c(\mathcal{P}_k)$ & Reference \\ \hline

		2 & $\chi_c(\mathcal{P}^*_2)=4$  & [Bipartite] & $\chi_c(\mathcal{P}_2)=8$ & [4CT] \\ \hline

		3 & $\chi_c(\mathcal{P}^*_3)=4$ & [4CT] & $ \chi_c(\mathcal{P}_3) \leq 6$ & \cite{NWZ21} \\ \hline

		4 & $\chi_c(\mathcal{P}^*_4)\cong 4$ & \cite{KNNW21+} & $\chi_c(\mathcal{P}_4)\leq 4$ & \cite{MRS16}\\ \hline

		5 & $\chi_c(\mathcal{P}^*_5)= 3$  &\cite{G59}, \cite{SY89}  & $\ast$ & \\ \hline

		6 & $\chi_c(\mathcal{P}^*_6) \leq 3$ &[this paper] & $\ast$ &\\ \hline

		7 & $\ast$ & & $\chi_c(\mathcal{P}_7)\leq 3$ & \cite{NSWX20+} \\ \hline

		8 & $\chi_c(\mathcal{P}^*_8)\cong \frac{8}{3}$ &\cite{NPW20+} & $\ast$ &\\ \hline

		11 &  $\chi_c(\mathcal{P}^*_{11})\leq \frac{5}{2}$ &\cite{DP17}, \cite{CL20} & $\ast$ &\\ \hline

		14 &  $\chi_c(\mathcal{P}^*_{14})\leq \frac{12}{5}$ &\cite{LWW21+}  & $\ast$ & \\ \hline

		17 & $\chi_c(\mathcal{P}^*_{17})\leq \frac{7}{3}$ &\cite{CL20}, \cite{PS21} & $\ast$ & \\ \hline

		$\cdots$ & $\cdots$ & $\cdots$ & $\cdots$ & \\ \hline

		$6p-2$ & $\chi_c(\mathcal{P}^*_{6p-2})\leq \frac{4p}{2p-1}$ &\cite{LNWZ21+} & $\chi_c(\mathcal{P}_{6p-2})\leq \frac{8p-2}{4p-3}$ &\cite{LNWZ21+} \\ \hline
		$6p-1$ & $\chi_c(\mathcal{P}^*_{6p-1})\leq \frac{4p}{2p-1}$ &\cite{LWZ20} & $\chi_c(\mathcal{P}_{6p-1})\leq \frac{4p}{2p-1}$ &\cite{LNWZ21+} \\ \hline
		$6p$ & $\ast$ & &$\chi_c(\mathcal{P}_{6p})< \frac{4p}{2p-1}$ &\cite{LNWZ21+} \\ \hline
		$6p+1$ & $\chi_c(\mathcal{P}^*_{6p+1})\leq \frac{2p+1}{p}$ &\cite{LTWZ13} & $\chi_c(\mathcal{P}_{6p+1})\leq \frac{8p+2}{4p-1}$  &\cite{LNWZ21+}\\ \hline
		$6p+2$ & $\ast$ &  & $\chi_c(\mathcal{P}_{6p+2})\leq \frac{2p+1}{p}$  &\cite{LNWZ21+}\\ \hline
		$6p+3$ & $\chi_c(\mathcal{P}^*_{6p+3})<\frac{2p+1}{p}$ &\cite{LWZ20} & $\chi_c(\mathcal{P}_{6p+3})<\frac{2p+1}{p}$ &\cite{LNWZ21+}\\ \hline
		\end{tabular}
\end{center}

In this table, when we write $\chi_c(\mathcal{C})=r$, it means that $\chi_c(\hat{G})\leq r$ for each member $\hat{G}$ of the class $\mathcal{C}$ and that the equality is known to hold for at least one member of the class. When we write $\chi_c(\mathcal{C})\cong r$, we mean that there is a sequence of signed graphs of $\mathcal{C}$ whose limit of the circular chromatic number is $r$. In such cases, sometimes it is verified that the $r$ is never reached by a single member of $\mathcal{C}$. For example, this is indeed the case for $\mathcal{P}^*_4$ as shown in \cite{KNNW21+}. In other cases, it is not known if the equality holds for some members or the inequality is always strict. In particular, for $\mathcal{P}^*_8$ the sequence that gives the limit of $\frac{8}{3}$ is $\{T^*_{2}(\Gamma_i)\}$ where $\Gamma_i$ is the sequence reaching the limit of $4$ for $\mathcal{P}^*_4$. It remains an open problem whether the equality can be reached in this case. 

There are some trivial inclusion among the classes considered here:  $\mathcal{P}^*_{k+2} \subseteq \mathcal{P}^*_{k}$ and $\mathcal{P}_{k+1} \subseteq \mathcal{P}_{k}$. In such cases, any upper bound for the larger class works also on the smaller one and any lower bound for the smaller one works on the larger one as well. In the entries of the table where we write $\ast$ the best known bounds come from the other entries of the table based on these inclusion.

To tight the gap in the bounds or, more ambitiously, to determine the exact values, is the subject of some of main work in the theory of coloring planar graphs. A notable conjecture is that of Jaeger-Zhang which can be restated as:

\begin{conjecture}{\rm [Jaeger-Zhang Conjecture]}\label{conj:JZ}
Given a positive integer $p$, we have
 $\chi_c(\mathcal{P}^*_{4p+1})\leq \frac{2p+1}{p}$. 
\end{conjecture}

A bipartite analogue of this conjecture was first proposed in \cite{NRS15}, but considering the result of \cite{NPW20+}, it is modified to the following.

\begin{conjecture}{\rm [Bipartite analogue of Jaeger-Zhang Conjecture]}\label{conj:JZ-BiP}
Given a positive integer $p$, we have
$\chi_c(\mathcal{P}^*_{4p})\leq \frac{4p}{2p-1}$. 
\end{conjecture}

The bound of Conjecture~\ref{conj:JZ-BiP} seems to what one may expect for $\chi_c(\mathcal{P}^*_{4p-1})$ and that of Conjecture~\ref{conj:JZ} seems to be what one may expect for $\chi_c(\mathcal{P}^*_{4p+2})$ as well.

\section{Conclusions and further questions}\label{sec:Conclusion}

In this paper, verifying the importance of the study of circular chromatic number of signed bipartite graphs we have presented the signed bipartite circular cliques. Then, using the 4-Color Theorem, we have shown an upper bound of $3$ for the circular chromatic number of signed bipartite planar graphs of negative girth at least $6$. We have provided a table summarizing the best known results on the circular chromatic number of signed planar graphs with a girth condition. Beside all the open questions that are summarized in the table, there are two questions of interest to mention.

The first is about the use the $4$-Color Theorem in our proof of the upper bound of $3$ for the circular chromatic number of the subclass $\chi_c(\mathcal{P}^*_{2})$. Could one find a relatively short proof of this without using the $4$-Color Theorem? Or can one show that, on the contrary, this result implies the $4$-Color Theorem? We recall that, in Section~\ref{sec:X_c}, reformulations of the $4$-Color Theorem using special classes of planar graphs of high girth are given. So this would not be a surprise. 

The second question is to potentially strengthen our result to include the Gr\"otzsch theorem as a special case. One possibility is observed by reformulating the Gr\"otzsch theorem itself as follows.

\begin{theorem}{\rm [Gr\"otzsch's theorem restated]}
If $G$ is a planar graph satisfying that $K_3\not\to G$, then $G\to K_3$.
\end{theorem}

We recall that if $K_3\not\to G$, then $S(K_3)\not\to S(G)$ and that if $G\to K_3$ then $S(G)\to S(K_3)$. Thus a potential strengthening of our result, which would include the Gr\"otzsch theorem, is as follows.

\begin{conjecture}
If $(G, \sigma)$ is a signed bipartite planar graph with the property that $S(K_3)\not \to (G, \sigma)$, then $\chi_c(G, \sigma)\leq 3$, i.e., $(G,\sigma)\to(K_{3,3}, M)$.
\end{conjecture}
	
We note that, as shown in \cite{NSWX20+}, for a signed bipartite graph $(G, \sigma)$ to have ``$(G,\sigma) \to (K_{3,3}, M)$'' is equivalent to have ``$(G,\sigma)\to(K_{6}, M)$'' where $(K_{6}, M)$ is the signed graph on $K_6$ with a perfect matching forming the set of negative edges.

Finally we would like to mention potential importance of the Table we have provided for the values of $\mathcal{P}^*_k$. This table, together with the recent development of the notion of circular chromatic number of signed graphs, helps to fill gaps in conjectures such as Jaeger-Zhang conjecture. When state alone there is gap of $4$ on the girth conditions between two consecutive statements (from $4p+1$ to $4p+5$). The refinement introduced here would help with easier inductive approach based on $k$. In such an approach, given an element $(G,\sigma)$ of $\mathcal{P}^*_{k+1}$, one would seek for an equivalent signature $\sigma'$ such that after contracting all negative edges of $\sigma'$, the resulting graph could be viewed as an element of $\mathcal{P}^*_k$. Then using a circular coloring, obtained by an inductive assumption on $k$, one may produce a required circular coloring of $(G, \sigma)$. This was indeed our approach to prove the upper bound of $3$ for the circular chromatic number of the class $\mathcal{P}^*_6$.

\medskip
{\bf Acknowledgement.} 
This work is supported by the French ANR project HOSIGRA (ANR-17-CE40-0022). It has also received funding from the European Union's Horizon 2020 research and innovation program under the Marie Sklodowska-Curie grant agreement No 754362.

\end{document}